\documentclass[reqno]{amsart}
\usepackage[english]{babel}
\usepackage{amssymb,verbatim}
\usepackage[T1]{fontenc}

\newtheorem{theorem}{Theorem}[section]

\newtheorem{lemma}[theorem]{Lemma}
\newtheorem{corollary}[theorem]{Corollary}

\theoremstyle{definition}

\newtheorem{example}[theorem]{Example}

\theoremstyle{remark}

\numberwithin{equation}{section}

\newcommand{\CEP}[1]{\mbox{$\mathbb{C}^{#1}$}}

\newcommand{\E}{\mbox{$\mathcal{E}$}}
\newcommand{\Eo}{\mbox{$\mathcal{E}_{0}$}}
\newcommand{\Ep}{\mbox{$\mathcal{E}_{p}$}}
\newcommand{\F}{\mbox{$\mathcal{F}$}}

\newcommand{\ddcn}[1]{\mbox{$\left(dd^{c}#1\right)^{n}$}}

\begin{document}

\author{Urban Cegrell}
\address{Department of Mathematics and Mathematical Statistics\\ Ume\aa \ University\\ SE-901 87 Ume\aa \\ Sweden}
\email{Urban.Cegrell@math.umu.se}
\title{Measures of finite pluricomplex energy}

\maketitle

\begin{abstract} In this note we study a complex Monge-Amp\`{e}re type equation of the 
form 
\[
(dd^cu)^n = \frac {ke^{-u}dV}{\int e^{-u}dV}\, .
\]
 
\end{abstract}

\begin{center}\bf
\today
\end{center}

\maketitle

\section{Introduction}
This paper is a revised version of \cite{cegrell_mt}.

Throughout this paper let $\Omega\subseteq\CEP{n}$, $n\geq 1$, be a bounded, connected, open, and 
hyperconvex set. Let $\Eo$, $\F_1$, and $\F$ be the energy classes introduced 
in~\cite{cegrell_pc,cegrell_gdm} (see section~2 for details). In~\cite{cegrell_pc}, the author 
proved the following theorem: 

\bigskip

\noindent {\bf Theorem~\ref{THM_DP}:} \emph{Let $\mu$ be a non-negative Radon measure. Then the following conditions are equivalent:}
\begin{enumerate}\itemsep2mm

\item \emph{there exists a function $u\in\E_1$ such that $\ddcn{u}=\mu$,}

\item \emph{there exists a constant $B>0$, such that}
\[
\int_{\Omega}(-\varphi)\, d\mu\leq B\left(\int_{\Omega}(-\varphi)\ddcn{\varphi}\right)^{\frac1{n+1}}\text{ for all } \varphi\in\Eo\, , 
\]
\emph{where $\ddcn{\cdot}$ is the complex Monge-Amp\`{e}re operator.}
\end{enumerate}

\bigskip

This theorem gives a  complete characterization of measures for which there exist a solution of the Dirichlet problem for the complex Monge-Amp\`{e}re operator in the class $\E_1$. 

In section 3 we give a direct proof of Theorem 3.1  without  use of the Rainwater lemma (see \cite{rain}). The solutions to the Dirichlet problem in Theorem 3.1 are
always unique. We are not going to discuss that point in this article. 

\bigskip

 In section 4, we prove

\bigskip

\noindent {\bf Theorem~4.6:} \emph{For every $k<(2n)^n$ there is a function $u\in\E_0\cap C$ with}
\[
(dd^cu)^n = \frac {ke^{-u}dV}{\int e^{-u}dV}\, ,
\]
\emph{where $dV$ is the normalized Lebesque measure on $\Omega$.}

\bigskip

The proof make use of Theorem~\ref{THM_DP},  
In section~\ref{sec_var}, we give an alternative proof  when k=1, where we use variational methods together with the following theorem:

\bigskip

\noindent {\bf Theorem~\ref{THM_ineq}:} \emph{To every $ b> {1/ (2n)^n}$ there exists a constant $B>0$, such that}
\[
\int_{\Omega} \text{\rm exp}(-u)\, dV\leq B\, \text{\rm exp} \left(b \int_{\Omega}(-u)\ddcn{u}\right)\quad \text{ for all }
u\in\E_1\, .
\]

\noindent{\bf Remark:} A stronger version of Theorem ~\ref{THM_ineq} is proved in \cite{bb_mt}.

\section{Preliminaries}
By
$\Eo$ we denote the family of all bounded plurisubharmonic functions $\varphi$ defined on $\Omega$ such that
\[
\lim_{z\to\xi} \varphi (z)=0\;\; \text{ for every }\;\; \xi\in\partial\Omega\, ,  \;\;\text{ and }\;\; \int_{\Omega} \ddcn{\varphi}<
\infty\, ,
\]
where $\ddcn{\,\cdot\,}$ is the complex Monge-Amp\`{e}re operator, normalized so that $dd^c = \frac{i}{ \pi}\partial  \bar \partial .$
Assume now $u$ that is a function such that there exists a decreasing sequence $\{u_j\}$, $u_j\in\Eo$, that
converges pointwise to $u$ on $\Omega$, as $j$ tends to $+\infty.$ For $p>0$, we say that

\begin{itemize}
\item{$u\in\mathcal F_p$, if
\[
\sup_{j\geq 1}\int_{\Omega}((-u_j)^p +1)\ddcn{u_j} < \infty\, ,
\]}

\item{$u\in\mathcal E_p$, if
\[
\sup_{j\geq 1}\int_{\Omega}(-u_j)^p\ddcn{u_j} < \infty\, ,
\]}

\item{$u\in\mathcal F$, if
\[
\sup_{j\geq 1}\int_{\Omega}\ddcn{u_j} < \infty\, .
\]}
\end{itemize}

The complex Monge-Amp\`{e}re operator is well-defined on these classes. See e.g.\cite{cegrell_pc,cegrell_gdm,czyz_cl,kol_mem} for more information about the energy classes.

\begin{theorem}\label{thm_holder2} Let $p \geq 1$, and $n\geq 2$. Then there exists a constant $D(n,p)\geq
1$, depending only on $n$ and $p$, such that for any $u_0,u_1,\ldots , u_n\in\Ep$ it holds that
\begin{multline*}
\int_\Omega (-u_0)^p dd^c u_1\wedge\cdots\wedge dd^c u_n \\ \leq D(n,p)\;
\left(\int_\Omega (-u_0)^p (dd^c u_0)^n\right)^{1/(n+p)} \cdots
\left(\int _\Omega (-u_n)^p( dd^c u_n)^n\right)^{1/(n+p)} \, .
\end{multline*}
Furthermore, $D(n,1)=1$ and $D(n,p)\geq1$ for $p\neq 1$.
\end{theorem}
\begin{proof}
This is Theorem~3.4 in~\cite{persson_ad} (see also~\cite{czyz_ineq,czyz_energy,cegrell_pc,cegr_pers}).
\end{proof}
It was proved in~\cite{czyz_ineq} (see also~\cite{czyz_modul}) that for
$p\neq 1$ the constant $D(n,p)$ in Theorem~\ref{thm_holder2} is strictly great than $1$.

\bigskip

The following variant is proved in ~\cite{cegrell_gdm}.

\begin{theorem}\label{thm_holder3} Let  $n\geq 2$. For any $u_0,u_1,\ldots , u_n\in\F$ it holds that
\[
\int_\Omega (-u_0) dd^c u_1\wedge\cdots\wedge dd^c u_n \leq
\left(\int_\Omega (-u_0) (dd^c u_1)^n\right)^{1/n} \cdots \left(\int _\Omega (-u_0)( dd^c u_n)^n\right)^{1/n}
\]
\end{theorem}

\section{Dirichlet's problem in  $\mathcal E_1$}\label{sec_DP}

\begin{theorem}\label{THM_DP}  Let $\mu$ be a non-negative Radon measure. Then the following conditions are equivalent:
\begin{enumerate}

\item there exists a function $u\in\E_1$ such that $\ddcn{u}=\mu$,

\item there exists a constant $B>0$, such that
\[
\int_{\Omega}(-\varphi)\, d\mu\leq B\left(\int_{\Omega}(-\varphi)\ddcn{\varphi}\right)^{\frac1{n+1}}\text{ for all } \varphi\in\Eo\, ,
\]
\end{enumerate}
\end{theorem}

\bigskip

This theorem gives a  complete characterization of measures for which there exist a solution of the Dirichlet problem for the complex Monge-Amp\`{e}re operator in the class $\E_1$. Originally, the theorem was proved by the author in ~\cite{cegrell_pc}.
The approximation theorem below is the main result in this section. It gives a direct proof of the
Dirichlet problem  without  use of the Rainwater lemma. The solutions to the Dirichlet problem at hand are
always unique. We are not going to discuss this here.

We say that a non-negative Radon measure $\mu$ belongs to $\mathcal M_1$ if there exists constant $A$ such that
\[
\int_{\Omega}(-u)\;d\mu\leq A\left(\int_{\Omega}(-u)\ddcn{u}\right)^{\frac{1}{n+1}}\, ,
\]
holds for all $u\in \Eo$. 

The setup: It is no loss of generality to assume that $\mu$ has compact support. So let $\mu\in \mathcal M_1$ with compact support. Let $\varphi$ be a usual regularization kernel and put $\mu_j = \varphi_j*\mu$ which is a well-defined non-negative compactly supported smooth function. Solve, using \cite{bt_in}, $\ddcn{u_j} = \mu_j$ for  $u_j \in\Eo.$ We show that this sequence converges to the solution of the Dirichlet problem.

\bigskip

\begin{theorem}  (Approximation theorem)\label{THM_approx}. With notations as above, $u_j$ converges as distributions and in $L^1(\mu)$ to a function $u\in\mathcal F_1$ and $\ddcn{u} = \mu$.
\end{theorem}
\begin{proof}
We claim that
\[
u=\lim_{j\to + \infty} (\sup_{k \geq j} u_k)^*\in \F_1 \qquad \text{ and } \ddcn{u} = \mu\, .
\]
For choose a weak*- convergent subsequence, again denoted by $u_j$ converging weak* to $u$. Then by the construction of $u_j$ we have that
\[
\int -u_j\ddcn{u_j} \leq \int -u_j\mu \leq  A\left(\int(-u_j)\ddcn{u_j}\right)^{\frac{1}{n+1}}
\]
so it follows from integration by parts
\[
\int - (\sup_{k \geq j} u_k)^*\ddcn{ (\sup_{k \geq j} u_k)^*} \leq \int -u_j\ddcn{u_j} \leq A^{(n+1)/n}\, .
\]
Note that integration by parts also gives the inequality
\[
\int v\mu \leq \int v\ddcn{u}
\]
for all negative plurisubharmonic functions $v$ (see~\cite{cegrell_weak}). Theorem 2.1 gives
\[
\int -u\mu = \lim \int -u\ddcn{u_j} \leq \left( \int -u\ddcn{u}\right)^{1/(n+1)}\lim \left(\int -u_j\ddcn{u_j}\right)^{n/n+1}
\]
and we will show that
\[
\lim \int -u_j\ddcn{u_j} \leq \int -u\mu
\]
so it follows that  $\int u\mu = \int u\ddcn{u}$. Let $j\leq k$:
\begin{multline*}
 \int -u_j\ddcn{u_j} \leq  \int -u_j\ddcn{u_k} \leq \\
\left(\int -u_j\ddcn{u_j}\right)^{1/(n+1)} \left(\int -u_k\ddcn{u_k}\right)^{n/n+1}
\end{multline*}
so  $ \int -u_j\ddcn{u_j}$ is monotonically increasing to some $c<+\infty.$ Theorem 2.2 gives
\begin{multline*}
 \int -u_j\ddcn{u_j} \leq   \int -u_k dd^c u_j\wedge (dd^cu_k)^{n-1} \leq \\
 \left(\int -u_k\ddcn{u_j}\right)^{1/n}\left(\int -u_k\ddcn{u_k}\right)^{(n-1)/n}\, .
\end{multline*}
Hence,
\[
\int -u_j\ddcn{u_j} \leq \left(\int -u_k\ddcn{u_j}\right)^{1/n} c^{(n-1)/n}\, .
\]
Now
\[
\lim_{k\to \infty}  \int -u_k\ddcn{u_j} =   \int \varphi_j*u\mu  = \int -u\ddcn{u_j}
\]
since $\varphi_j*u_k$ tends uniformly to $ \varphi_j*u, j\to \infty$ on the support of $\mu$ and we have that $ \int u\mu = \int u\ddcn{u}$ since  $\lim \int -u\ddcn{u_j} = \int -u\mu$ by construction. Let now $v\in\Eo$ be given. We have for t>0
\begin{multline*}
\int -(u+tv)\mu =\lim \int -(u+tv)\ddcn{u_j} \leq \\
\left(\int -(u+tv)\ddcn{u+tv}\right)^{1/(n+1)}\lim \left(\int -u_j\ddcn{u_j}\right)^{n/n+1} =\\
\left(\int -(u+tv)\ddcn{u+tv}\right)^{1/(n+1)} \left(\int -u\ddcn{u}\right)^{n/n+1} \leq\\
\left(\int -(u+tv)\ddcn{u+tv}\right)^{1/(n+1)} +  \left(\int -u\ddcn{u}\right)^{n/n+1}\, .
\end{multline*}
This is an inequality between two polynomials in $t$. The polynomials are equal at $t = 0$ so the coefficients for their first order terms satisfy the same inequality. The coefficient for the left hand side is $\int-v\mu$ and for the right hand side $\int -v\ddcn{u}.$ Therefore $\int-v\mu \leq \int -v\ddcn{u}$ for every $\int -v\ddcn{u}$ and since we already know the opposite inequality we have proved that $\mu = \ddcn{u}.$ The function $u$ is uniquely determined so the original sequence was already weak*-convergent to $u$. This completes the proof.
\end{proof}

\begin{corollary} Let $\mu$ be a positive measure with $\mu (\Omega)<+\infty$ and $\mu(P)=0$ for every pluripolar set. Then there is a uniquely determined function $u\in\F$ with $(dd^cu)^n = \mu.$
\end{corollary}
It follows from Rainwater's lemma that $\mu = f(dd^cv)^n$ for a $v\in\F_1$ so the solutions $u_j$ to $(dd^cu_j)^n = \min (f,j)(dd^cv)^n$ decreases to $u, j \to +\infty.$

\section{Compact and convex sets in $\mathcal E_1$}
We consider $\F(\Omega)$ as a convex cone in $L^1(\Omega,dV)$.The  Theorems 2.1 and 2.2 gives on
\[
\F: \left(\int (dd^c(u+v))^n\right)^{\frac1{n}} \leq \left(\int (dd^cu)^n\right)^{\frac1{n}} + \left(\int (dd^cv)^n\right)^{\frac1{n}}
\]
and on
\[
\E_1: \left(\int -(u+v)(dd^c(u+v))^n\right)^{\frac1{n+1}} \leq \left(\int -u(dd^cu)^n\right)^{\frac1{n+1}} + \left(\int -v(dd^cv)^n\right)^{\frac1{n+1}}\, .
\]
Therefore,
\[
\left\{u\in\F ; \int(dd^cu)^n \leq C\right\} \qquad \text{ and } \qquad \left\{u\in\E_1 ; \int-u(dd^cu)^n \leq C\right\}
\]
are convex in $\F$ and $\E_1$, resp. Both are also compact.

\begin{lemma} Assume $u_j, u \in\E_1$ and $\sup\int-u_j(dd^cu_j)^n < + \infty.$ If $u_j\to u, j\to +\infty$ as distributions, then  $u_j\to u, j\to +\infty$ in $L^1((dd^cw)^n)$
for every $w\in\E_1$.
\end{lemma}
Note: It may happen that $\int-u_j(dd^cu_j)^n = 1$ but  $u_j\to 0,  j\to +\infty$ as distributions.
\begin{proof}
Let $m>0$. Then
\begin{multline*}
|u - u_j| \leq |u - \max(u,mw) + \max(u,mw) - \max(u_j,mw) + \max(u_j,mw) - u_j|\\ \leq
\max(u,mw) -u + |\max(u,mw) - \max(u_j,w)| + \max(u_j,mw) - u_j|
\end{multline*}
so
\begin{multline*}
\int |u - u_j|(dd^cw)^n  \leq \int(\max(u,mw) -u)(dd^cw)^n \\ + \int |\max(u,mw) - \max(u_j,w)|(dd^cw)^n +
 \int \max(u_j,mw) - u_j)(dd^cw)^n\, .
\end{multline*}
At the right hand side, when $m \to + \infty,$ the first integral tends to 0 by monotone convergence and that the second tends to 0 when $j \to 0$ follows from Lemma 1.4 in \cite{ce_ko}. We use Theorems 3.1 and 2.1 to estimate the third term:
\begin{multline*}
\int (\max(u_j,mw) - u_j)(dd^cw)^n \leq \int_{\{u_j < mw\}} -u_j(dd^cw)^n \\ = \int -u_j\chi_{\{u_j<mw\}}(dd^cw)^n \leq
\left(\int -u_j(dd^cu_j)^n\right)^{\frac{1}{n+1}}\left(\int -w\frac{u_j}{mw}(dd^cw)^n\right)^{\frac{n}{n+1}} \\ \leq \frac{c}{m^{\frac{n}{n+1}}} \to 0,\qquad \text{as } m \to +\infty\, .
\end{multline*}
\end{proof}

We need the following two theorems.

\begin{theorem}\cite{ce_pe} There exist a constant C such that $sup |u - v| \leq C|| g - h||^{\frac1 n}_2$ where $u,v\in\mathcal F$ and $(dd^cu)^n = gdV, (dd^cv)^n = hdV$ and $f,g \in L^2(dV)$.
\end{theorem}

\begin{theorem}\cite{A_p}, \text{Theorem B.}
There exist a uniform constant $a_n>0$, depending only on $n$,  such that for any positive number $0 \leq \mu < n$ and any  $u \in  \mathcal F (\Omega)$ such that $\int_{\Omega} (dd^c u)^n \leq \mu^n, $ we have that
\begin{equation}
\int_{\Omega} e^{- 2 u} d V  \leq  \left(\pi^n +  a_n \frac{\mu}{(n - \mu)^n}\right)   \delta _\Omega^{2 n},
\end{equation}
where $V$ is the $2 n$-dimensional Lebesgue measure on $\CEP{n}$ and $\delta _\Omega$ is the diameter of $\Omega$.
\end{theorem}

We prove

\begin{theorem}\label{THM_ineq} To every $ b> {1/ (2n)^n}$ there exists a constant $B>0$, such that
\[
\int_{\Omega} \text{\rm exp}(-u)\, dV\leq B\, \text{\rm exp} \left(b \int_{\Omega}(-u)\ddcn{u}\right)\quad \text{ for all }
u\in\E_1\, ,
\]
\end{theorem}
\begin{proof}
Set
\[
a = \int_{\Omega}(-u)\ddcn{u}\, .
\]
Then
\begin{multline*}
\ddcn{u} =  \chi_{ \{u > -ab\}}\ddcn{u} + \chi_{ \{u \leq -ab\}}\ddcn{u}\\=
\chi_{ \{u > -ab\} }\ddcn{\max(u,-ab)} + \chi_{ \{u \leq -ab\}}\ddcn{u}\\ \leq
\ddcn{\text{max}(u,-ab)} + \chi_{ \{u \leq -ab\}}\ddcn{u}\, .
\end{multline*}
Solve $\ddcn{w} = \chi_{ \{u \leq -ab \}}\ddcn{u}$ for $w\in\mathcal F$.
By Theorem 4.5 in \cite{cegrell_pc}
\[
u \geq \text{max}(u,-ab) + w\,
\]
and since
\[
\int \ddcn{w} \leq {a/ab}  <  (2n)^n
\]
it follows from Theorem 4.3 that
\[
\int \text{\rm exp}(-u)dV \leq D\, \text{\rm exp} (ab)
\]
and the proof is complete.
\end{proof}
\

\begin{theorem} (Consequence of Schauder's fixed point theorem)
Suppose A is a convex and compact subset of $\E_1.$ If $ T:A \to A$ is a continuous map then there is $u\in A$ with $u = T(u).$
\end{theorem}

\bigskip

\begin{theorem}\label{THM_DP2} For every $k<(2n)^n$ there is a function $u\in\E_0\cap C$ with 
\[
(dd^cu)^n = \frac {ke^{-u}dV}{\int e^{-u}dV}\, , 
\]
where $dV$ is the normalized Lebesque measure on $\Omega$.
\end{theorem}
\begin{proof}
We wish to use the fixed point theorem and define $B = \{ u\in\mathcal F, \int (dd^cu)^n \leq k\}$. Using Theorem 4.3 and Corollary 3.3 we can consider the map
$ u: B \to T(u)\in B$ where $T(u)$ is the unique function in $\F$ with
\[
(dd^cT(u))^n =  \frac {ke^{-u}dV}{\int e^{-u}dV}\, .
\]

\noindent Choose m such that $(\frac{m}{m-1})^n k< (2n)^n$. By Theorem 4.3, there is a constant c such that
\[
\int-T(u)(dd^c T(u))^n = k \int-T(u)e^{-u}dV/\int e^{-u}dV \leq 
\]
\[
k(\int (-T(u))^mdV)^\frac{1}{m}( \int (e^{-\frac{m}{m-1}u}dV)^{\frac{m-1}{m}}\leq (m!)^{\frac{1}{ m}}kc \qquad \text{ for all }\  u\in\text{B} .
\]
Hence it follows that $T(u)\in\mathcal F_1$ and Theorem 4.3 and Theorem 4.2 now gives  that $T(u) \in \mathcal E_0\cap C.$

We restrict $T$ to the convex and compact set  
\[
A=\left\{u\in\mathcal F, \int (dd^cu)^n \leq k, \int-u(dd^c u)^n \leq (m!)^{\frac{1}{m}}kc \right\}\, .
\]
Then $A \to T(u)\in A$ and it remains to show that $T$ is continuous on $A.$ It is then enough to prove that
if $u_j, u \in A$ and $u_j\to u, j\to +\infty$ as distributions, then
\[
\frac {e^{-u_j}dV}{\int e^{-u_j}dV} \to \frac {e^{-u}dV}{\int e^{-u}dV}\qquad \text{as } j\to +\infty \text{ in } L^2(dV)\, .
\]
Choose $t>1$ so that $kt^n < n^n$ and $p, \frac{1}{t} + \frac{1}{p} = 1$ and define $w_j =\sup\limits_{k\geq j} (u_k)^*.$ Now, \[
\int |e^{-u_j} - e^{-u}|^{2}dV \leq 2\int |e^{-u_j} - e^{-w_j}|^{2}dV + 2\int |e^{-u} - e^{-w_j}|^{2}dV
\]
and
\begin{multline*}
\int |e^{-u_j} - e^{-w_j}|^{2}dV = \int e^{-2u_j}|1 - e^{u_j - w_j}|^{2}dV \\
\leq \left(\int e^{-2tu_j}\right)\left(\int |1 - e^{u_j - w_j}|dV\right)^{\frac{1}{p}} \leq q\left(\int (w_j - u_j)dV\right)^{\frac{1}{p}} \to 0\quad \text{ as } j \to + \infty 
\end{multline*}
by Lemma 4.1. The constant $q$ can be estimated using Theorem 4.4.
\end{proof}

\section{The variational method}\label{sec_var}
Let $u,v\in \E_1$, and assume that $v$ is continuous. For $t < 0$, put $P(u+tv)=\sup\{w\in \E_1:w\leq u+tv\}$.
Then $P(u+tv)\in \E_1$ (see~\cite{czyz_pp}). Write
\[
e(u) = \int -u(dd^cu)^n
\]
and let $ J: \ A \to R$ be a continuous functional on
$\mathcal E_1.$ Put $F(u) = \frac 1{n+1} e(u) + J(u).$ If
\[
\liminf_{ t\to 0^-}\frac{J(P(u+tv)) - J(u)}{t} \geq \liminf_{t\to 0^+}\frac{J((u+tv) - J(u)}{t}
\]
for all $v\in\E_0 \cap C$ and if $u_m$ is a minimum point of $F$, then
\[
\int -v(dd^cu_m)^n + J'(u_m+tv)|_{t=0} =0 \qquad\text{ for all } v\in\E_0\cap C\, .
\]

\begin{example} Take $J(u) = -\log\int e^{-u}dV.$ Then $J$ is defined on $\mathcal E_1$ and  Theorem 4.4
shows that we only have to minimize over a compact convex subset. A calculation shows that
\[
\liminf_{ t\to 0^-}\frac{J(P(u+tv)) - J(u)}{t} \geq\lim _{t\to 0^+}\frac{J(u+tv) - J(u)}{t} =
\frac{\int ve^{-u}dV}{\int e^{-u}dV}
\]
for all $v\in\E_0\cap C$ so
\[
-(dd^cu_m)^n + \frac{e^{-u_m}dV}{\int e^{-u_m}dV} =0\, .
\]
For let $v\in\E_0\cap C$. Then, for $t<0$ we have that
\begin{multline*}
\frac{J(P(u+tv)) - J(u)}{t} = \frac{\log\frac{\int e^{-P(u+tv)dV}}{\int e^{-u}dV}}{-t} \\=
\frac{\log \left(1 + \frac{\int e^{-P(u+tv)} - e^{-u}dV}{\int e^{-u}dV}\right)}{-t} \geq\left( \frac{\int e^{-P(u+tv)}
- e^{-u}dV}{\int e^{-u}dV}= \frac{\int e^{-u}(e^{-P(u+tv)+ u} - 1)dV}{\int e^{-u}dV}\right) \\
\geq\frac{\log \left(1 + \frac{\int e^{-u}(-P(u+tv) + u)dV}{\int e^{-u}dV}\right)}{-t} =
\frac{\log \left(1 + \frac{\int e^{-u}(-P(u+tv) + u+tv)dV}{\int e^{-u}dV} -\frac{\int e^{-u}tvdV}{\int e^{-u}dV}\right)}{-t}
\\ \geq\frac{\log \left(1 - \frac{\int e^{-u}tvdV}{\int e^{-u}dV}\right)}{-t} \to \frac{\int e^{-u}vdV}{\int e^{-u}dV}
\qquad  \text{ as } t \to 0\, .
\end{multline*}
\hfill{$\Box$}
\end{example}

\end{document}